\documentclass[12pt,one column]{IEEEtran}
\usepackage{setspace}
\pdfoutput=1
\usepackage{amsmath,mathtools}
\usepackage{amssymb}
\usepackage{amsthm}
\usepackage{bbm}
\usepackage{dsfont}
\usepackage{cite}
\usepackage{color}
\usepackage{epsfig}
\usepackage{epsf}
\usepackage{rotating}
\usepackage{mathrsfs}
\usepackage{epsfig}
\usepackage{graphics}
\usepackage{enumerate}
\usepackage{euscript}
\usepackage{accents}
\DeclareMathAccent{\wtilde}{\mathord}{largesymbols}{"65}
\usepackage[T3,T1]{fontenc}
\DeclareSymbolFont{tipa}{T3}{cmr}{m}{n}
\DeclareMathAccent{\invbreve}{\mathalpha}{tipa}{16}
\usepackage[multiple]{footmisc}
\usepackage{anyfontsize}
\usepackage{t1enc}
\usepackage{amssymb}
\usepackage{cite}
\usepackage{color}
\usepackage{epsfig}
\usepackage{epsf}
\usepackage{rotating}
\usepackage{mathrsfs}
\usepackage{epsfig}
\usepackage{graphics}
\usepackage{subfigure}

\newcommand{\vertiii}[1]{{\left\vert\kern-0.25ex\left\vert\kern-0.25ex\left\vert #1 
    \right\vert\kern-0.25ex\right\vert\kern-0.25ex\right\vert}}

\theoremstyle{plain}
\newtheorem{thm}{Theorem}
\newtheorem{lem}{Lemma}

\begin{document}
\vskip 1cm

\thispagestyle{empty} \vskip 1cm


\title{{On The Absolute Constant \\in
Hanson-Wright Inequality }}
\author{ Kamyar Moshksar\\
\small Columbia College\\ Vancouver, BC, Canada} \maketitle

\begin{abstract} 
We revisit and slightly modify the proof of the Gaussian Hanson-Wright inequality where we keep track of the absolute constant in its formulation.
\end{abstract}

In this short report we investigate the following concentration of measure inequality which is the "Gaussian-version" of the Hanson-Wright inequality\footnote{In the general formulation of the Hanson-Wright inequality, the entries of $\underline{\boldsymbol{x}}$ are i.i.d. random variables with zero mean, unit variance and sub-Gaussian tail decay.}  [1], [2]:
\begin{thm}[Hanson-Wright Inequality]
\label{lem11}
Let $\underline{\boldsymbol{x}}\sim \mathrm{N}(\underline{0}_n,I_n)$. If $A$ is a nonzero $n\times n$ matrix, then 
\begin{eqnarray}
\label{HW1}
\Pr\Big(|\underline{\boldsymbol{x}}^TA\underline{\boldsymbol{x}}-\mathbb{E}[\underline{\boldsymbol{x}}^TA\underline{\boldsymbol{x}}]|\ge a\Big)\leq 2\exp\Big(-C\min\Big\{\frac{a^2}{\|A\|_2^2},\frac{a}{\|A\|}\Big\}\Big),
\end{eqnarray}
for every $a>0$ where $C$ is an absolute constant that does not depend on $n$, $A$ and $a$.  
\end{thm}
Here, $\|A\|_2$ is the Hilbert-Schmidt norm of $A$ defined by 
$$
\|A\|_2:=(\mathrm{tr}(A^TA))^{\frac{1}{2}}
$$
and $\|A\|$ is the operator norm of $A$ defined by 
$$
\|A\|:=\max_{\|\underline{x}\|_2\leq 1}\|A\underline{x}\|_2=(\lambda_{\max}(A^TA))^{\frac{1}{2}},
$$
where $\lambda_{\max}(M)$ denotes the largest eigenvalue of a square matrix $M$.  

Let us denote the largest value for the absolute constant $C$ in (\ref{HW1}) by $C_{HW}$. Even a lower bound on $C_{HW}$ is not reported in the literature. The following lemma presents a lower bound on $C_{HW}$ in the special case where the matrix $A$ in (\ref{HW1}) is a real symmetric matrix.
\begin{lem}
\label{hey_hey11}
For every $0<r<1$ define
\begin{eqnarray}
\label{xi_11}
f(r):=\sum_{j=0}^\infty\frac{r^j}{j+2}.
\end{eqnarray}
If the matrix $A$ in Theorem~\ref{lem11} is real symmetric, then the absolute constant $C_{HW}$ is bounded as  
\begin{eqnarray}
\label{kappa_11}
C_{HW}\ge C^*:=\max_{0<r<1}\min\Big\{\frac{r}{4},\frac{1}{8f(r)}\Big\}\approx 0.1457.
\end{eqnarray}
\end{lem}
\begin{proof}
See the Appendix.
\end{proof}
\begin{figure}
   \centering
    \includegraphics[width=0.7\textwidth]{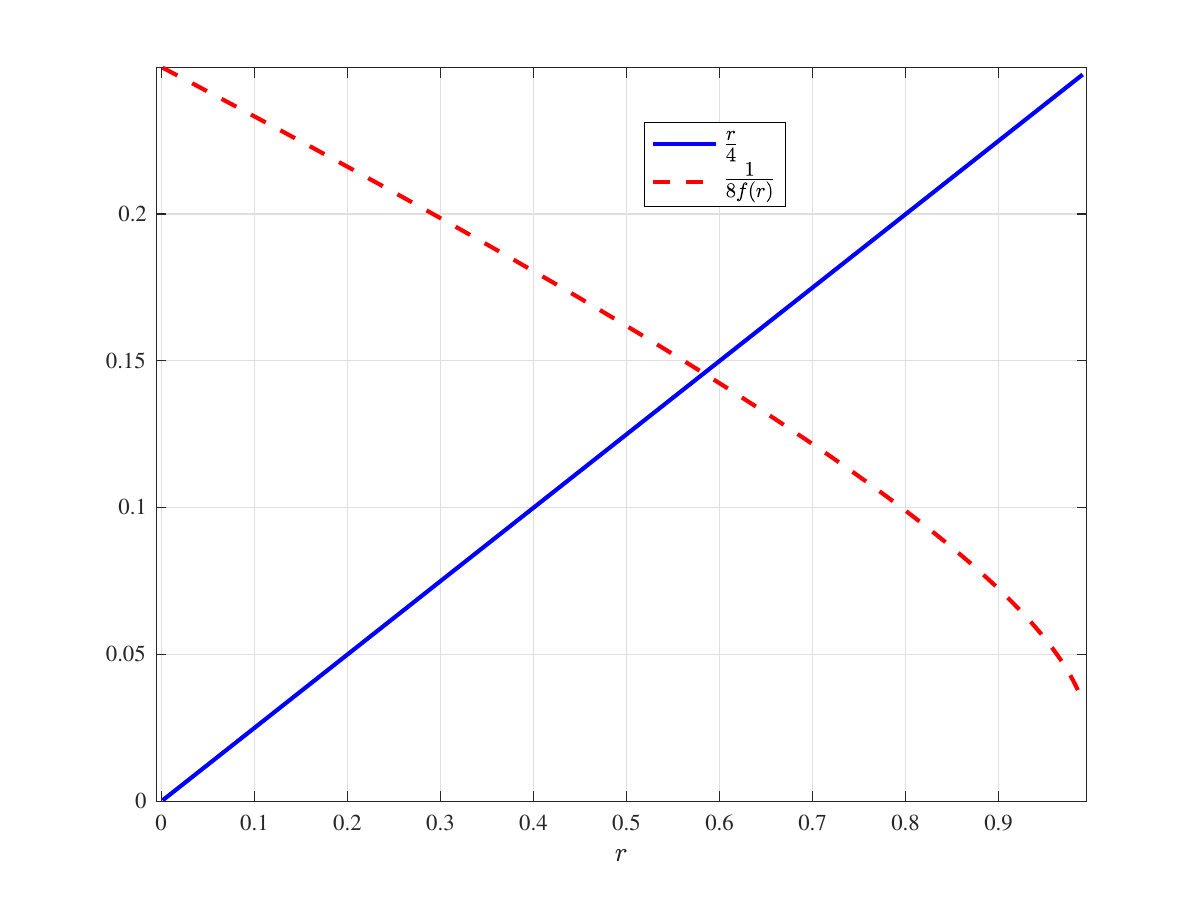}
     \caption{Plots of $\frac{r}{4}$ and $\frac{1}{8f(r)}$ in terms of $0<r<1$. We see that $\min\{\frac{r}{4},\frac{1}{8f(r)}\}$ is maximized for $r\approx 0.583$ and its maximum value is approximately $0.1457$. }
    \label{pict_111}
\end{figure}
Fig.~\ref{pict_111} shows plots of $\frac{r}{4}$ and $\frac{1}{8f(r)}$ in terms of $0<r<1$. We see that the maximum of $\min\{\frac{r}{4},\frac{1}{8f(r)}\}$ is approximately $0.1457$.

\textit{Remark}- The function $f(r)$ in (\ref{xi_11}) is a special case of the so-called Lerch Transcendent function.

\section*{Appendix}
Since $A$ is real symmetric, we can write $A=U^T\Lambda U$ where $U$ is an orthogonal matrix and $\Lambda$ is a diagonal matrix with real diagonal entries $\lambda_1,\cdots, \lambda_n$. Define $\underline{\boldsymbol{v}}:=U\underline{\boldsymbol{x}}$. Then $\underline{\boldsymbol{v}}\sim \mathrm{N}(\underline{0}_n,I_n)$ and we have  $\underline{\boldsymbol{x}}^TA\underline{\boldsymbol{x}}=\sum_{i=1}^n\lambda_i\boldsymbol{v}_i^2$ and $\mathbb{E}[\underline{\boldsymbol{x}}^TA\underline{\boldsymbol{x}}]=\sum_{i=1}^n\lambda_i$. Fix $0<r<1$ and let 
\begin{eqnarray}
\label{ass_111}
0<s\leq\frac{r}{2\max_i|\lambda_i|}.
\end{eqnarray}
Applying Markov's inequality, 
\begin{eqnarray}
\label{kenar_11}
\Pr(\underline{\boldsymbol{x}}^TA\underline{\boldsymbol{x}}-\mathbb{E}[\underline{\boldsymbol{x}}^TA\underline{\boldsymbol{x}}]>a)&\leq& \frac{1}{e^{sa}}\mathbb{E}\big[e^{s(\underline{\boldsymbol{x}}^TA\underline{\boldsymbol{x}}-\mathbb{E}[\underline{\boldsymbol{x}}^TA\underline{\boldsymbol{x}}])}\big]\notag\\
&\stackrel{}{=}&e^{-sa}\prod_{i=1}^n\mathbb{E}[e^{s\lambda_i(\boldsymbol{v}_i^2-1)}]\notag\\
&=&e^{-sa}\prod_{i=1}^n\frac{e^{-s\lambda_i}}{\sqrt{1-2s\lambda_i}},
\end{eqnarray} 
where the penultimate step is due to independence of $\boldsymbol{v}_1,\cdots, \boldsymbol{v}_n$ and last step is due to $|s\lambda_i|<\frac{1}{2}$ and the fact that $\mathbb{E}[e^{s\boldsymbol{v}^2}]=(1-2s)^{-\frac{1}{2}}$ for every $s<\frac{1}{2}$ and $\boldsymbol{v}\sim \mathrm{N}(0,1)$. We write 
\begin{eqnarray}
\label{kenar_22}
\frac{1}{\sqrt{1-2s\lambda_i}}&=&e^{-\frac{1}{2}\ln(1-2s\lambda_i)}\notag\\&\stackrel{(a)}{=}&e^{\frac{1}{2}\sum_{j=1}^\infty\frac{(2s\lambda_i)^j}{j}}\notag\\
&=&e^{s\lambda_i+\frac{1}{2}\sum_{j=2}^\infty\frac{(2s\lambda_i)^j}{j}}\notag\\
&=&e^{s\lambda_i+\frac{1}{2}\sum_{j=0}^\infty\frac{(2s\lambda_i)^{j+2}}{j+2}}\notag\\
&=&e^{s\lambda_i+2s^2\lambda_i^2\sum_{j=0}^\infty\frac{(2s\lambda_i)^{j}}{j+2}}\notag\\
&\stackrel{(b)}{\leq}&e^{s\lambda_i+2s^2\lambda_i^2\sum_{j=0}^\infty\frac{r^{j}}{j+2}}\notag\\
&\stackrel{(c)}{=}&e^{s\lambda_i}e^{2s^2f(r)\lambda_i^2},
\end{eqnarray}
where $(a)$ is due to $-\ln(1-x)=\sum_{j=1}^\infty\frac{x^j}{j}$ for $|x|<1$, $(b)$ is due to $|2s\lambda_i|\le r$ and $f(r)$ in $(c)$ is defined in (\ref{xi_11}). By (\ref{kenar_11}) and (\ref{kenar_22}), 
\begin{eqnarray}
\label{big_11}
\Pr(\underline{\boldsymbol{x}}^TA\underline{\boldsymbol{x}}-\mathbb{E}[\underline{\boldsymbol{x}}^TA\underline{\boldsymbol{x}}]>a)\leq \exp\Big(-sa+2s^2f(r)\sum_{i=1}^n\lambda_i^2\Big).
\end{eqnarray}
Minimizing $-sa+2s^2f(r)\sum_{i=1}^n\lambda^2_i$ in terms of $s$ subject to the constraint in (\ref{ass_111}), we get 
\begin{eqnarray}
s=\min\Big\{\frac{a}{4f(r)\|A\|_2^2}, \frac{r}{2\|A\|}\Big\},
\end{eqnarray}
where we have replaced $\sum_{i=1}^n\lambda_i^2$ and $\max_i |\lambda_i|$ by $\|A\|_2^2$ and $\|A\|$, respectively.  We consider two cases: 
\begin{enumerate}
  \item If $\frac{r}{2\|A\|}\leq \frac{a}{4f(r)\|A\|_2^2}$ or equivalently, $2rf(r)\|A\|_2^2\leq \|A\|a$, then $s=\frac{r}{2\|A\|}$ and we get  
 \begin{eqnarray}
 \label{big_22}
 -sa+2s^2f(r)\|A\|_2^2&=&-\frac{ra}{2\|A\|}+2\Big(\frac{r}{2\|A\|}\Big)^2f(r)\|A\|_2^2\notag\\
 &=&-\frac{ra}{2\|A\|}+\frac{2r^2f(r)\|A\|_2^2}{4\|A\|^2}\notag\\
 &\leq&-\frac{ra}{2\|A\|}+\frac{ra}{4\|A\|}\notag\\
 &=& -\frac{ra}{4\|A\|},
\end{eqnarray}
where the penultimate is due to the condition $2rf(r)\|A\|_2^2\leq \|A\|a$. 
  \item If $\frac{r}{2\|A\|}> \frac{a}{4f(r)\|A\|_2^2}$, then $s=\frac{a}{4f(r)\|A\|_2^2}$ and we get
 \begin{eqnarray}
 \label{big_33}
 -sa+2s^2f(r)\|A\|_2^2=-\frac{a^2}{8f(r)\|A\|_2^2}.
\end{eqnarray}
\end{enumerate}
By (\ref{big_11}), (\ref{big_22}) and (\ref{big_33}), it follows that 
\begin{eqnarray}
\label{dum_11}
\Pr(\underline{\boldsymbol{x}}^TA\underline{\boldsymbol{x}}-\mathbb{E}[\underline{\boldsymbol{x}}^TA\underline{\boldsymbol{x}}]>a)&\leq&\exp\Big(\max\Big\{-\frac{a^2}{8f(r)\|A\|_2^2},-\frac{ra}{4\|A\|} \Big\}\Big)\notag\\
&=& \exp\Big(-\min\Big\{\frac{a^2}{8f(r)\|A\|_2^2},\frac{ra}{4\|A\|} \Big\}\Big).
\end{eqnarray}
Replacing the matrix $A$ by $-A$ and following the same lines of reasoning, we see that $\Pr(\underline{\boldsymbol{x}}^TA\underline{\boldsymbol{x}}-\mathbb{E}[\underline{\boldsymbol{x}}^TA\underline{\boldsymbol{x}}]<-a)$ is bounded from above by the same expression on the right side of (\ref{dum_11}). Hence,  
\begin{eqnarray}
\label{mew_11}
\Pr(|\underline{\boldsymbol{x}}^TA\underline{\boldsymbol{x}}-\mathbb{E}[\underline{\boldsymbol{x}}^TA\underline{\boldsymbol{x}}]|>a)&\leq& 2\exp\Big(-\min\Big\{\frac{a^2}{8f(r)\|A\|_2^2},\frac{ra}{4\|A\|} \Big\}\Big)\notag\\
&\leq&2\exp\Big(-\min\Big\{\frac{r}{4},\frac{1}{8f(r)}\Big\}\min\Big\{\frac{a^2}{\|A\|_2^2},\frac{a}{\|A\|}\Big).
\end{eqnarray}
The upper bound in (\ref{mew_11}) holds for any $0<r<1$. Minimizing this bound over $r$, we get 
\begin{eqnarray}
\Pr(|\underline{\boldsymbol{x}}^TA\underline{\boldsymbol{x}}-\mathbb{E}[\underline{\boldsymbol{x}}^TA\underline{\boldsymbol{x}}]|>a)\leq 2\exp\Big(-C^*\min\Big\{\frac{a^2}{\|A\|_2^2},\frac{a}{\|A\|} \Big\}\Big),
\end{eqnarray}
where $C^*$ is given in (\ref{kappa_11}).

\end{document}